\documentclass[11pt,runningheads]{llncs}
\usepackage{amsmath}
\usepackage{verbatim}
\usepackage{graphicx}
\usepackage[utf8]{inputenc}
\usepackage{tikz}
\usepackage{multicol}
\usepackage{bm}
\usepackage[utf8]{inputenc}
\usepackage{amsfonts}
\usepackage{xcolor}

\definecolor{blue}{HTML}{1F77B4}
\definecolor{orange}{HTML}{FF7F0E}
\definecolor{green}{HTML}{2CA02C}
\newtheorem{assumption}{Assumption}

\newcommand{\x}{\textbf{x}}
\newcommand{\y}{\textbf{y}}

\begin{document}
\title{Convergence of Physics-informed Neural Networks for  Fully Nonlinear PDE's }
%
%

\author{Avetik Arakelyan\inst{1,2} \and
Rafayel Barkhudaryan\inst{1,2}}

%
%
\institute{Institute of Mathematics, NAS of Armenia, Yerevan, Armenia\\
\email{avetik.arakelyan@ysu.am}\\ \and
Yerevan State University, Yerevan, Armenia\\
\email{barkhudaryan@ysu.am}}
\maketitle              
\begin{abstract}
The present work is focused on exploring  convergence  of Physics-informed Neural Networks (PINNs) when applied to a specific class of second-order fully nonlinear Partial Differential Equations (PDEs). It is well-known that as the number of data grows, PINNs generate a sequence of minimizers which correspond to a sequence of neural networks. We  show that such sequence converges to a unique viscosity solution of a certain class of  second-order fully nonlinear PDE's, provided the latter satisfies the comparison principle in the viscosity sense. 
\end{abstract}

\section{Introduction}

PDEs play a crucial role in many fields of engineering and fundamental science ranging from fluid dynamics to acoustic and structural engineering. Finite elements modeling (FEM) methods are the standard solvers employed ubiquitously in the industry \cite{strang2008analysis,reddy1993introduction}. Despite their popularity, FEM methods display some limitations such as their computational cost for large industrial problems (mainly due to the required mesh size) and issues in leveraging external data sources, such as sensors data, to drive the solution of the PDEs. 

The PINNs approach discussed below is regarded as a promising alternative to FEM methods for covering some of these limitations. This approach is quite different from the standard supervised ML. In fact, instead of relying purely on data, it uses the physical properties of the PDE itself to guide the training process. Known data points can be easily added on top of the physics-based loss function to speed-up training speed.
Recently, interesting results have been obtained from comparing the FEM and PINN methods \cite{grossmann2024can}.

Machine learning techniques using deep neural networks have been successfully applied in various fields \cite{Lecun_Nature15_DeepLearning}
such as computer vision and natural language processing.
A notable advantage of using neural networks 
is its efficient implementation
using a dedicated hardware
(see \cite{Lagaris_98_ANN-ODE-PDE,Darbon_19_NNsHJ}).
Such techniques have also been applied in 
solving partial differential equations
 \cite{Raissi_19_PINNs,Lagaris_98_ANN-ODE-PDE,Dissanayake_94_ANN-PDE,Lagaris_00_ANN-Irregular,Berg_18_Unified,Sirignano_JCP18_DGM}, 
and it has become a new sub-field under the name of Scientific Machine
Learning (SciML) \cite{Baker_19_DCworkshop,Lu_19_Deepxde}.
The term Physics-Informed Neural Networks (PINNs) was introduced in \cite{Raissi_19_PINNs} and it has become one of the most popular deep learning methods in SciML. 
PINNs employ a neural network as a solution surrogate
and seek to find the best neural network guided by data and physical laws expressed as PDEs. Following the original formulation in \cite{Raissi_19_PINNs}, we give a brief overview of physics-informed neural networks  in the context of solving  partial differential equations. Generally, we consider PDEs taking the form
\begin{align}
\label{eq: PDE}
     \mathbf{u}_t +  \mathcal{N}[\mathbf{u}] = 0, \ \  t \in [0, T],  \ \mathbf{x} \in \Omega,
\end{align} 
subject to the initial and boundary conditions
\begin{align*}
     &\mathbf{u}( 0, \mathbf{x})=\mathbf{g}(\mathbf{x}), \ \ \mathbf{x} \in \Omega, \\
     &\mathcal{B}[\mathbf{u}] = 0,  \ \   t\in [0, T], \  \mathbf{x} \in  \partial \Omega,
\end{align*}
where $\mathcal{N}[\cdot]$ is a linear or nonlinear differential operator, and  $\mathcal{B}[\cdot] $ is a boundary operator  corresponding to Dirichlet, Neumann, Robin, or periodic boundary conditions. In addition, $\mathbf{u}$ describes the unknown latent solution that is governed by the  PDE system of Equation \eqref{eq: PDE}. 

We proceed by representing the unknown solution $\mathbf{u}(t, \mathbf{x})$ by a deep neural network $\mathbf{u}_{\mathbf{\theta}}(t, \mathbf{x})$, where $\mathbf{\theta}$ denotes all tunable parameters of the network (e.g., weights and biases). This allows us to define the PDE residuals as
\begin{align}
    \label{eq: pde_residual}
    \mathcal{R}_{\mathbf{\theta}}(t, \mathbf{x}) = \frac{\partial \mathbf{u}_{\mathbf{\theta}}}{\partial t}(t_r, \mathbf{x}_r) + \mathcal{N}[\mathbf{u}_{\mathbf{\theta}}](t_r, \mathbf{x}_r)
\end{align}

Then, a physics-informed model can be trained by minimizing the following composite loss function
\begin{align*}
    \label{eq: PINN_loss}
    \mathcal{L}(\mathbf{\theta}) =  \mathcal{L}_{ic}(\mathbf{\theta}) +  \mathcal{L}_{bc}(\mathbf{\theta}) +   \mathcal{L}_r(\mathbf{\theta}), 
\end{align*}
where 
\begin{align*}
     &\mathcal{L}_{ic}(\mathbf{\theta}) = \frac{1}{N_{ic}} \sum_{i=1}^{N_{ic}} \left| \mathbf{u}_{\mathbf{\theta}}(0, \mathbf{x}_{ic}^i) - \mathbf{g}(\mathbf{x}_{ic}^i) \right|^2, \\
     &\mathcal{L}_{bc}(\mathbf{\theta}) = \frac{1}{N_{bc}} \sum_{i=1}^{N_{bc}} \left| \mathcal{B}[\mathbf{u}_{\mathbf{\theta}}]( t_{bc}^i, \mathbf{x}_{bc}^i) \right|^2, \\
    &\mathcal{L}_r(\mathbf{\theta}) = \frac{1}{N_r} \sum_{i=1}^{N_r} \left| \mathcal{R}_{\mathbf{\theta}}(t^i_r, \mathbf{x}^i_r) \right|^2. 
\end{align*}
Here $\{\mathbf{x}_{ic}^i\}_{i=1}^{N_{ic}}$, $\{t_{bc}^i, \mathbf{x}_{bc}^i\}_{i=1}^{N_{bc}}$ and $\{t_{r}^i, \mathbf{x}_{r}^i\}_{i=1}^{N_{r}}$ can be the vertices of a fixed mesh or points that are randomly sampled at each iteration of a gradient descent algorithm. Notice that all required gradients with respect to input variables or network parameters $\mathbf{\theta}$ can be efficiently computed via automatic differentiation \cite{griewank2008evaluating}.

A series of works have shown the effectiveness of PINNs in one, two or three dimensional problems: fractional PDEs \cite{Pang_SISC19_fPINNs,Song_19_fPINNs,ren2023class}, stochastic differential equations \cite{Zhang_19_SPDE_PINNs,HanE_18_DLSPDEs}, 
biomedical problems \cite{Raissi_Nature20_HFM}, and fluid mechanics \cite{Mao_20_HighSpeedFlows,cai2021physics,eivazi2024physics}.

\section{Setting of the Problem and Notations}

For the setting of the problem and notations we  mainly follow the work \cite{shin2020convergence}.
Let $U$ be a bounded domain (open and connected) in $\mathbb{R}^d$, and for twice differentiable function $u:U\to\mathbb R$ let $Du$ and $D^2 u$ denote the gradient and Hessian matrix of $u$, respectively. It is noteworthy that $\overline{U},$ which is the closure of $U$ will be compact. Also let the function $F(x, r, p, X)$ be a continuous real-valued function defined on $U\times\mathbb{R}\times\mathbb{R}^d\times S^d$, with $S^d$ being the space of real symmetric $d\times d$ matrices. Write
$$
\mathcal F[u](x)\equiv F\left(x, u(x), Du(x), D^2u(x)\right).
$$
We consider partial differential equations (PDEs) of the form 

\begin{equation}\label{PDE}
\begin{cases}
\mathcal{F}[u](\bm{x}) =  0\quad \forall \bm{x} \in U,\\
 u(\bm{x}) = g(\bm{x}) \quad \forall \bm{x} \in \Gamma = \partial U.
\end{cases}
\end{equation}


The goal is to approximate the solution to the PDE \eqref{PDE}
from a set of training data. 
The training data consist of two types of data sets: residual and initial/boundary data.
A residual datum is a vector $\bm{x}_r$,
where $\bm{x}_r \in U$
and an initial/boundary datum is a pair of input and output $(\bm{x}_{b}, g(\bm{x}_{b}))$, where $\bm{x}_{b} \in \Gamma$.
The set of $m_r$ residual input data points and the set of $m_{b}$ initial/boundary input data points are denoted by
$\mathcal{T}_r^{m_r} = \{\bm{x}_r^i\}_{i=1}^{m_r}$ 
and $\mathcal{T}_{b}^{m_{b}} = \{\bm{x}_{b}^i\}_{i=1}^{m_{b}}$, respectively.
Let us denote the vector of the number of training samples by $\bm{m}=(m_r,m_{b})$.
Note that we slightly abuse notation as in \cite{Raissi_19_PINNs}: $\bm{x}_r$ refers to a point in $U$.
Similarly, $\bm{x}_{b}$ refers to a point in $\Gamma$.
$m_r$ and $m_{b}$ represent the number of training data points in $U$ and $\Gamma$, respectively. 
\begin{remark}
    The present paper only considers the high regularity setting for PDEs 
    where point-wise evaluations 
    are well-defined.
    We also assume $u$ is a scalar real-valued function.
\end{remark}

Given a class of neural networks $\mathcal{H}_{n}$ 
that may depend on the number of training samples $\bm{m}$ (also may implicitly depend on the data sample itself),
we seek to find a neural network $h^*$ in $\mathcal{H}_{n}$ that minimizes 
an objective (loss) function.
Here $n$ represents a complexity of the class, e.g., the number of parameters.
To define an appropriate objective function, 
let us consider a loss integrand, which is the standard choice in physics informed neural networks (PINNs) \cite{Raissi_19_PINNs}:
\begin{equation} \label{def:loss}
\begin{split}
\textbf{L}(\bm{x}_r,{\bm{x}}_b;h,\bm{\lambda},\bm{\lambda}^R) &= \left(\lambda_r\|\mathcal{F}[h](\bm{x}_r)\|^2 \right)\mathbb{I}_{U}(\bm{x}_r) + \lambda_r^R R_r(h) \\
&+ 
\lambda_{b}\|h(\bm{x}_{b}) - g(\bm{x}_{b})\|^2 \mathbb{I}_{\Gamma}(\bm{x}_{b})
+ \lambda_{b}^R R_{b}(h),
\end{split}
\end{equation}
where $\|\cdot\|$ is the Euclidean norm, 
$\mathbb{I}_{A}(\bm{x})$ is the indicator function on the set $A$,
$\bm{\lambda} = (\lambda_r, \lambda_{b})$, 
$\bm{\lambda}^R = (\lambda_r^R, \lambda_{b}^R)$,
and $R_r(\cdot), R_{b}(\cdot)$ are regularization functionals.
Here $\lambda_r, \lambda_r^R,\lambda_{b}, \lambda_{b}^R \in \mathbb{R}_{+}\cup\{0\}$.

Suppose $\mathcal{T}_r^{m_r}$ and $\mathcal{T}_{b}^{m_{b}}$ are independently and identically distributed (iid) samples from probability distributions $\mu_r$ and $\mu_{b}$, respectively.
Let us define the empirical probability distribution on $\mathcal{T}_r^{m_r}$ 
by
$\mu_r^{m_r} = \frac{1}{m_r}\sum_{i=1}^{m_r} \delta_{\bm{x}_r^i}$.
Similarly, $\mu_{b}^{m_{b}}$ is defined.
The empirical loss and the expected loss are obtained by taking the expectations 
on the loss integrand \eqref{def:loss} with respect to $\mu^{\bm{m}}=\mu_r^{m_r}\times \mu_{b}^{m_{b}}$ and 
$\mu = \mu_r \times \mu_{b}$, respectively:
\begin{equation} \label{def:empirical-expected-loss}
	\begin{split}
	\text{Loss}_{\bm{m}}(h;\bm{\lambda},\bm{\lambda}^R) = \mathbb{E}_{\mu^{\bm{m}}}[\textbf{L}(\bm{x}_r,\bm{x}_b;h,\bm{\lambda},\bm{\lambda}^R)], 
	\quad
	\text{Loss}(h;\bm{\lambda},\bm{\lambda}^R) = \mathbb{E}_{\mu}[\textbf{L}(\bm{x}_r,\bm{x}_b;h,\bm{\lambda},\bm{\lambda}^R)].
	\end{split}
\end{equation}
In order for the expected loss to be well-defined,
it is assumed that $\mathcal{F}[h]$ is in $L^2(U;\mu_r)$,
and $h$ and $g$ are in $L^2(\Gamma;\mu_{b})$
for all $h\in \mathcal{H}_{n}$.
If the expected loss function were available, 
its minimizer would be the solution to the PDE \eqref{PDE} or close to it.
However, since it is unavailable in practice,
the empirical loss function is employed.
This leads to the following minimization problem:
\begin{equation} \label{def:problem}
\min_{h \in \mathcal{H}_{n}} \text{Loss}_{\bm{m}}(h;\bm{\lambda},\bm{\lambda}^R).
\end{equation}
We then hope a minimizer of the empirical loss to be close to the solution to the PDE \eqref{PDE}.

We remark that in general, global minimizers to the problem of \eqref{def:problem} need not exist.
However, for $\epsilon > 0$, there always exists a 
$\epsilon$-suboptimal global solution $h^{\epsilon} \in \mathcal{H}_{n}$ \cite{Houska_19_Global}
satisfying $\text{Loss}_{\bm{m}}(h^{\epsilon}) \le \inf_{h \in \mathcal{H}_{n}} \text{Loss}_{\bm{m}}(h) + \epsilon$.
All the results of the present paper remain valid
if one replace global minimizers 
to $\epsilon$-suboptimal global minimizers for sufficiently small $\epsilon$.
Henceforth, for the sake of readability, we assume the existence of at least one global minimizer of the minimization problem \eqref{def:problem}.
Note that for $\bm{\lambda} \le \bm{\lambda'}$ 
and
$\bm{\lambda}^R \le \bm{\lambda'}^{R}$ 
(element-wise inequality), 
we have
$\text{Loss}_{\bm{m}}(h;\bm{\lambda},\bm{\lambda}^R) \le \text{Loss}_{\bm{m}}(h;\bm{\lambda}',\bm{\lambda'}^{R})$.

When $\bm{\lambda}^R = 0$, we refer to the (either empirical or expected) loss as the (empirical or expected) PINN loss \cite{Raissi_19_PINNs}. 
We denote the empirical and the expected PINN losses by $\text{Loss}^{\text{PINN}}_{\bm{m}}(h;\bm{\lambda})$
and 
$\text{Loss}^{\text{PINN}}(h;\bm{\lambda})$, respectively.

Thus, the following losses will be in our focus:
\begin{align}
\label{def:PINN-loss}
\text{Loss}_m^{\text{PINN}}(h;\boldsymbol{\lambda})& = \frac{\lambda_r}{m_r}\sum_{i=1}^{m_r}||\mathcal F[h](\boldsymbol{x_r^i})||^2 + \frac{\lambda_b}{m_b}\sum_{i=1}^{m_b}||h(\boldsymbol{x_b^i})-g(\boldsymbol{x_b^i})||^2,\\
\label{def:PINN-loss-2}
\text{Loss}^{\text{PINN}}(h;\bm{\lambda}) &= \lambda_r  \|\mathcal{F}[h]\|^2_{L^2(U;\mu)}
+\lambda_{b} \|h - g\|^2_{L^2(\Gamma;\mu_{b})}.
\end{align}
Here $m=(m_r,m_b).$

\subsection{Degenerate Elliptic Equations and Viscosity Solutions}
In this section we recall the notion of viscosity solutions defined for a degenerate elliptic PDEs. It was first introduced by Crandall and Lions (see \cite{MR690039} and \cite{MR732102}) for first order Hamilton-Jacobi equations. It turns out that this notion is an effective tool also in the study of second order (elliptic and parabolic) fully nonlinear problems. There is a vast literature devoted to viscosity solutions by now, and for a general theory the reader is referred to \cite{MR1118699}, \cite{MR1351007} and references therein. Below for the sake of convenience we invoke the definition of a viscosity solution and one of the main tools the so-called comparison principle which implies the uniquness of a viscosity solution.

Let $\Omega$ be an open subset of $\mathbb{R}^n$, and for twice differentiable function $u:\Omega\to\mathbb R$ let $Du$ and $D^2 u$ denote the gradient and Hessian matrix of $u$, respectively. Also let the function $F(x, r, p, X)$ be a continuous real-valued function defined on $\Omega\times\mathbb{R}\times\mathbb{R}^n\times S^n$, with $S^n$ being the space of real symmetric $n\times n$ matrices. Denote
$$
\mathcal F[u](x)\equiv F\left(x, u(x), Du(x), D^2u(x)\right).
$$
We consider the following second order fully nonlinear partial differential equation:
\begin{equation}\label{pde_visc}
\mathcal F[u](x) = 0, \qquad   x\in\Omega.
\end{equation}

\begin{definition}
The equation \eqref{pde_visc} is \textbf{degenerate elliptic} if
$$
 F(x, r, p,X)\leq F(x, s, p, Y )  \quad\mbox{whenever }\quad  r\leq s  \quad\mbox{and }\quad Y \leq X,
$$
where $Y \leq X$ means that $X - Y$ is a nonnegative definite symmetric matrix.
\end{definition}

\begin{definition}
$u:\Omega\to\mathbb R$ is called a \textbf{viscosity subsolution} of \eqref{pde_visc}, if it is upper semicontinuous and for each $\varphi\in C^2(\Omega)$ and local maximum point $x_0\in \Omega$ of $u-\varphi$ we have
\begin{equation}\label{2D-def-Visc-Subsol}
F\left(x_0, u(x_0), D\varphi(x_0), D^2 \varphi(x_0)\right)\le 0.
\end{equation}
\end{definition}

\begin{definition}
$u:\Omega\to\mathbb R$ is called a \textbf{viscosity supersolution} of \eqref{pde_visc}, if it is lower semicontinuous and for each $\varphi\in C^2(\Omega)$ and local minimum point $x_0\in \Omega$ of $u-\varphi$ we have
$$
F\left(x_0, u(x_0), D\varphi(x_0), D^2 \varphi(x_0)\right)\ge 0.
$$
\end{definition}
\begin{definition}
$u:\Omega\to\mathbb R$ is called a \textbf{viscosity solution} of \eqref{pde_visc}, if it is both a viscosity subsolution and supersolution (and hence continuous) for \eqref{pde_visc}.
\end{definition}
\begin{assumption}[Comparison Principle] 
We say the system \eqref{pde_visc} satisfies comparision principle if for every viscosity subsolution $u$ and viscosity supersolution $v$ of a system \eqref{pde_visc}, such that $u(x)\leq v(x)$ on $x\in \partial\Omega$ implies $u(x)\leq v(x)$ for all $x\in \Omega.$
\end{assumption}
In the rest of the work we will assume that a system \eqref{PDE} is degenerate elliptic and satisfies above comparison principle. Note that it will easily imply the uniqueness of a viscosity solution for a fixed boundary condition. Thus, the system \eqref{PDE} has a unique viscosity solution provided it is degenerate elliptic and satisfies comparison principle.

In fact we stressed on the comparison principle to be satisfied not only for the uniqueness of a viscosity solution, but also the existence of such solutions. The following well-known result due to Ishii is true \cite{MR1118699}:
\begin{theorem}[Existence via Perron's Method]
  Let comparison principle hold for \eqref{PDE}; i.e., given $u$ viscosity subsolution
and $v$ viscosity supersolution satisfying the same boundary condition, then $u \leq v.$
Suppose also that there exist $\underline{u}$ and $\overline{u}$ which are, respectively, a viscosity
subsolution and a viscosity supersolution, satisfying the same boundary
condition. Then
\[W(x)=\sup\{w(x)\;|\; \underline{u}\leq w\leq\overline{u},\mbox{and}\; w\; \mbox{viscosity subsolution}\}\]
is a viscosity solution of \eqref{PDE}.
\end{theorem}


\subsection{Function Spaces and Regular Boundary}
In this section we adopt the notation from \cite{Gilbarg_15_EllipticPDEs,Friedman_08_ParabolicPDEs}.
Let $U$ be a bounded domain in $\mathbb{R}^d$. 
Let $\bm{x}=(x_1,\cdots,x_d)$ be a point in $\mathbb{R}^d$.
For a positive integer $k$, 
let $C^{k}(U)$ be the set of functions having all derivatives of order $\le k$ continuous in $U$. 
Also, let $C^{k}(\overline{U})$ be the set of functions in $C^{k}(U)$ whose derivatives of order $\le k$ have continuous extensions to $\overline{U}$ (the closure of $U$).

We call a function $u$ uniformly H\"{o}lder continuous with exponent $\alpha$ in $U$ if the quantity
\begin{equation} \label{def:holder-coefficient}
[u]_{\alpha;U}=\sup_{x,y \in U, x\ne y} \frac{|u(x) - u(y)|}{\|x-y\|^\alpha} < +\infty, \qquad 0 < \alpha \le 1,
\end{equation}
is finite.
Also, we call a function $u$ locally H\"{o}lder continuous with exponent $\alpha$ in $U$ if $u$ is uniformly H\"{o}lder continuous with exponent $\alpha$ on compact subsets of $U$. $[u]_{\alpha;U}$ is called the H\"{o}lder constant (coefficient) of $u$ on $U$.

Given a multi-index $\textbf{k}=(k_1,\cdots,k_d)$, we define
\begin{align*}
D^\textbf{k} u = \frac{\partial^{|\textbf{k}|}u}{\partial x_1^{k_1} \cdots \partial x_d^{k_d}},
\end{align*}
where $|\textbf{k}| = \sum_{i=1}^d k_i$.
Let $\textbf{k}=(k_1,\cdots,k_d)$ and $\textbf{k}'=(k'_1,\cdots,k'_d)$.
If $k_j' \le k_j$ for all $j$, we write $\textbf{k}' \le \textbf{k}$.
Let us define
\begin{align*}
[u]_{j,0;U} &:= \sup_{|\textbf{k}|=j} \sup_{U} |D^{\textbf{k}}u|, \qquad j=0,1, 2\cdots, \\
[u]_{j,\alpha;U} &:= \sup_{|\textbf{k}|=j} [D^{\textbf{k}}u]_{\alpha;U} =  \sup_{|\textbf{k}|=j} \left[\sup_{x,y \in U, x\ne y}  \frac{\|D^{\textbf{k}}u(x) - D^{\textbf{k}}u(y)\|}{\|x-y\|^\alpha}\right].
\end{align*}

\begin{definition}
	For a positive integer $k$
	and $0 < \alpha \le 1$, 
	the H\"{o}lder spaces $C^{k,\alpha}(\overline{U})$ ($C^{k,\alpha}(U)$) are 
	the subspaces of $C^{k}(\overline{U})$ ($C^{k}(U)$) consisting of 
	all functions $u \in C^{k}(\overline{U})$ ($C^{k}(U)$)
	satisfying $\sum_{j=0}^{k} [u]_{j,0;U} + [u]_{k,\alpha;U} < \infty$.
\end{definition}
For simplicity, we often write $C^{k,0} = C^{k}$ and $C^{0,\alpha} = C^\alpha$ for $0 < \alpha < 1$.
The related norms are defined on $C^{k}(\overline{U})$ and $C^{k,\alpha}(\overline{U})$, respectively, by 
\begin{align*}
\|u\|_{C^{k}(\overline{U})} =  \sum_{j=0}^{k} [u]_{j,0;U}, \qquad
\|u\|_{C^{k,\alpha}(\overline{U})} = \|u\|_{C^{k}(\overline{U})} + [u]_{k,\alpha;U}, \qquad 0 < \alpha \le 1.
\end{align*}
With these norms, $C^{k}(\overline{U})$ and $C^{k,\alpha}(\overline{U})$ are Banach spaces. 
Also, we denote $\frac{\partial u}{\partial x_j}$ as $D_ju$
and $\frac{\partial^2 u}{\partial x_i \partial x_j}$ as $D_{ij}u$.

\begin{definition}
	A bounded domain $U$ in $\mathbb{R}^d$ and its boundary are said to be of class $C^{k,\alpha}$ where $0\le \alpha \le 1$, if
	at each point $x_0 \in \partial U$
	there is a ball $B=B(x_0)$ and a one-to-one mapping
	$\psi$ of $B$ onto $D \subset \mathbb{R}^d$ such that 
	(i) $\psi(B\cap U) \subset \mathbb{R}_{+}^d$,
	(ii) $\psi(B\cap \partial U) \subset \partial \mathbb{R}_{+}^d$,
	(iii) 
	$\psi \in C^{k,\alpha}(B), \psi^{-1} \in C^{k,\alpha}(D)$.
\end{definition}

\subsection{Neural Networks}
Let $h^L: \mathbb{R}^{d} \to \mathbb{R}^{d_{\text{out}}}$ 
be a feed-forward neural network having $L$ layers and $n_\ell$ neurons in the $\ell$-th layer. 
The weights and biases in the $l$-th layer are represented by 
a weight matrix $\mathbf{W}^l \in \mathbb{R}^{n_l \times n_{l-1}}$ 
and a bias vector $\mathbf{b}^l \in \mathbb{R}^{n_l}$, respectively. 
Let $\bm{\theta}_L:=\{\bm{W}^j, \bm{b}^j\}_{1\le j \le L}$.
For notational completeness, let $n_0 = d$ and $n_L = d_{\text{out}}$.
For a fixed positive integer $L$, let $\vec{\bm{n}} = (n_0,n_1,\cdots,n_L) \in \mathbb{N}^{L+1}$ where $\mathbb{N}=\{1,2,3, \cdots \}$.
Then, $\vec{\bm{n}}$ describes a network architecture.
Given an activation function $\sigma(\cdot)$ which is applied element-wise,
the feed-forward neural network is defined by
\begin{align*}
	h^{\ell}(\textbf{x}) &= \bm{W}^{\ell}\sigma(h^{\ell-1}(\textbf{x})) + \bm{b}^{\ell}
	\in \mathbb{R}^{N_\ell}, 
	\qquad \text{for} \quad 2 \le \ell \le L
\end{align*}
and $h^1(\textbf{x}) = \bm{W}^{1}\textbf{x} + \bm{b}^{1}$.
The input is $\mathbf{x} \in \mathbb{R}^{n_0}$, 
and the output of the $\ell$-th layer is $h^\ell(\bm{x}) \in \mathbb{R}^{n_\ell}$. 
Popular choices of activation functions include 
the sigmoid ($1/(1+e^{-x})$), the hyperbolic tangent ($\tanh(x)$),
and the rectified linear unit ($\max\{x,0\}$).
Note that $h^L$ is called a $(L-1)$-hidden layer neural network or a $L$-layer neural network.

Since a network $h^L(\bm{x})$ depends on the network parameters $\bm{\theta}_L$
and the architecture $\vec{\bm{n}}$,
we often denote $h^L(\bm{x})$ by $h^L(\bm{x};\vec{\bm{n}}, \bm{\theta}_L)$. 
If $\vec{\bm{n}}$ is clear in the context, we simple write $h^L(\bm{x};\bm{\theta}_L)$.
Given a network architecture, we define a neural network function class
\begin{equation} \label{def:NN-class}
	\mathcal{H}_{\vec{\bm{n}}}^{\text{NN}} = \left\{h^L(\cdot;\vec{\bm{n}},\bm{\theta}_L):\mathbb{R}^{d}\mapsto \mathbb{R}^{d_\text{out}} | \bm{\theta}_L = \{(\bm{W}^j,\bm{b}^j) \}_{j=1}^L
	  \right\}.
\end{equation}

Since $h^L(\bm{x};\bm{\theta}_L)$ is parameterized by $\bm{\theta}_L$, 
the problem of \eqref{def:problem} with $\mathcal{H}_{\vec{\bm{n}}}^{\text{NN}}$ is
equivalent to
\begin{equation*} 
	\min_{\bm{\theta}_L}  \text{Loss}_{\bm{m}}(\bm{\theta}_L;\bm{\lambda},\bm{\lambda}^R), \quad
	\text{where} \quad
	\text{Loss}_{\bm{m}}(\bm{\theta}_L;\bm{\lambda},\bm{\lambda}^R) = \text{Loss}_{\bm{m}}(h( x;\bm{\theta}_L);\bm{\lambda},\bm{\lambda}^R).
\end{equation*}
Throughout this paper, the activation function is assumed to be sufficiently smooth,
which is common in practice.
Note that $\tanh(x)$ is bounded by 1 and is 1-Lipschitz continuous. 
The $k$-th derivative of $\tanh(x)$
is expressed as a polynomial of $\tanh(x)$ with a finite degree,
which shows both the boundedness and Lipschitz continuity.

\begin{remark}
    In what follows, 
    we simply write 
    $\mathcal{H}_{\vec{\bm{n}}}^{NN}$
    as
    $\mathcal{H}_{\bm{m}}$, 
    assuming $\vec{\bm{n}}$ depends on $\bm{m}$
    and possibly implicitly on the data samples itself.
    Since neural networks
    are universal approximators,
    the network architecture $\vec{\bm{n}}$ is expected to 
    grow 
    proportionally on $\bm{m}$.
\end{remark}

\section{Convergence of PINN}\label{sec:analysis}


The convergence analysis is based on the probabilistic space filling arguments \cite{calder2019consistency,Finlay_18_Lipschitz}.
In this regard, we recall the following assumptions given in \cite{shin2020convergence} on the training data distributions.
\begin{assumption} \label{assumption:data-dist}
	Let $U$ be a bounded domain in $\mathbb{R}^d$
	that is at least of class $C^{0,1}$ 
	and $\Gamma =\partial U$.
	Let $\mu_r$ and $\mu_{b}$ be probability distributions defined on $U$ and $\Gamma$, respectively.
	Let $\rho_r$ be the probability density of $\mu_r$ with respect to $d$-dimensional Lebesgue measure on $U$.
	Let $\rho_{b}$ be the probability density of $\mu_{b}$
	with respect to $(d-1)$-dimensional Hausdorff measure on $\Gamma$.
	\begin{itemize}
		\item[$\bullet$]  $\rho_r$ and $\rho_{b}$ are supported on $\overline{U}$
		and $\Gamma$, respectively.
		Also, $\inf_{U} \rho_r > 0$ and $\inf_{\Gamma} \rho_{b} > 0$.
		\item[$\bullet$] 
		For $\epsilon > 0$, there exists partitions of $U$ and $\Gamma$, $\{U_j^\epsilon\}_{j=1}^{K_{r}}$
		and $\{\Gamma_j^\epsilon\}_{j=1}^{K_{b}}$
		that depend on $\epsilon$
		such that 
		for each $j$, 
		there are cubes $H_{\epsilon}(\textbf{z}_{j,r})$ and 
		$H_{\epsilon}(\textbf{z}_{j,b})$ of side length $\epsilon$
		centered at $\textbf{z}_{j,r} \in U_j^\epsilon$
		and $\textbf{z}_{j,b} \in \Gamma_j^\epsilon$, respectively, 
		satisfying $U_j^\epsilon \subset H_{\epsilon}(\textbf{z}_{j,r})$
		and $\Gamma_j^\epsilon \subset H_{\epsilon}(\textbf{z}_{j,b})$.
	  \item[$\bullet$]  
		There exists positive constants $c_r, c_{b}$ such that 
		$\forall \epsilon > 0$,
		the partitions from the above satisfy
		$c_r \epsilon^{d} \le \mu_r(U_j^\epsilon)$
		and
		$c_{b} \epsilon^{d-1} \le 
		\mu_{b}(\Gamma_j^\epsilon)$
		for all $j$.

		There exists positive constants $C_r, C_{b}$ such that 
		$\forall \x_r \in U$ and $\forall \x_b \in \Gamma$,
		$\mu_r(B_{\epsilon}(\x_r) \cap U) \le C_r\epsilon^d$
		and 
		$\mu_{b}(B_\epsilon(\x_{b}) \cap \Gamma) \le C_{b} \epsilon^{d-1}$
		where 
		$B_\epsilon(\x)$ is a closed ball of radius $\epsilon$ centered at $\x$.

		Here $C_r, c_r$ depend only on $(U,\mu_r)$
		and $C_{b}, c_{b}$ depend only on $(\Gamma, \mu_{b})$.
		\item[$\bullet$]  When $d=1$, 
		we assume that all boundary points are available.
		Thus, no random sample is needed on the boundary.
	\end{itemize}
\end{assumption}

We remark that Assumption~\ref{assumption:data-dist} guarantees that random samples drawn from probability distributions can fill up both the interior of the domain $U$ and the boundary $\partial U$. 
These are mild assumptions and can be satisfied in many practical cases. 
For example, let $U=(0,1)^d$. Then the uniform probability distributions on both $U$ and $\partial U$ satisfy Assumption~\ref{assumption:data-dist}.


We now state similar result as in  \cite{shin2020convergence} that bounds the expected PINN loss in terms of a regularized empirical loss.  
Let us recall that $\bm{m}$ is the vector of the number of training data points, i.e.,
$\bm{m} = (m_r,m_{b})$.
The constants $c_r, C_r, c_{b}, C_{b}$ are introduced in Assumption~\ref{assumption:data-dist}.
For a function $u$, $[u]_{\alpha;U}$ is the H\"{o}lder constant of $u$ with exponent $\alpha$ in $U$ \eqref{def:holder-coefficient}.
\\\\
In the next coming theorems the following technical lemma plays a crucial role.

\begin{lemma}[Lemma B.2 in \cite{shin2020convergence}] \label{cor-sampling}
	Let $X$ be a compact subset in $\mathbb{R}^d$.
	Let $\mu$ be a probability measure supported on $X$.
	Let $\rho$ be the probability density of $\mu$ with respect to 
	$s$-dimensional Hausdorff measure on $X$
	such that $\inf_{X} \rho > 0$.
	Suppose that 
	for $\epsilon > 0$, 
	there exists a partition of $X$, $\{X_k^\epsilon\}_{k=1}^{K_{\epsilon}}$
	that depends on $\epsilon$ 
	such that for each $X_k^\epsilon$, 
	$c\epsilon^s\le \mu(X_k^\epsilon)$ where $c > 0$ depends only on $(\mu, X)$,
	and 
	there exists a cube $H_{\epsilon}(\textbf{z}_k)$ of side length $\epsilon$ centered at some $\textbf{z}_k$ in $X_k$
	such that $X_k \subset H_{\epsilon}(\textbf{z}_k)$.
	Then,
	with probability at least $1 - \sqrt{n}(1-1/\sqrt{n})^n$ over iid $n$ sample points $\{\x_i\}_{i=1}^n$ from $\mu$,
	for any $\x \in X$, there exists a point $\x_j$
	such that $\|\x - \x_j\| \le \sqrt{d}c^{-\frac{1}{s}}n^{-\frac{1}{2s}}$.
\end{lemma}

\begin{theorem} \label{thm:gen}
Suppose Assumption~\ref{assumption:data-dist} holds.
	Let $m_r$ and $m_{b}$ be the number of iid samples from $\mu_r$ and $\mu_{b}$, respectively.
	For some $0 < \alpha \le 1$, 
	let $h, g$, $R_r(h)$, $R_{b}(h)$  satisfy
	\begin{align*}
	    \big[\mathcal{F}[h]\big]^2_{\alpha;U} \le R_r(h)  < +\infty, \quad
	    \big[h\big]_{\alpha;\Gamma}^2 \le R_{b}(h) < +\infty, \quad
	  \big[g\big]_{\alpha;\Gamma}  < +\infty.
	\end{align*}
	Let $\bm{\lambda} = (\lambda_r,\lambda_{b})$ be a fixed vector.
    Let $\bm{\hat{\lambda}}_{\bm{m}}^R = (\hat{\lambda}_{r,\bm{m}}^R,\hat{\lambda}_{b,\bm{m}}^R)$
    be a vector whose elements to be defined.

	For $d \ge 2$, with probability at least, 
	$(1 - \sqrt{m_r}(1-1/\sqrt{m_r})^{m_r})(1 - \sqrt{m_{b}}(1-1/\sqrt{m_{b}})^{m_{b}})$,
	we have
	\begin{equation*}
	    \text{Loss}^{\text{PINN}}(h;\bm{\lambda})
        \le C_{\bm{m}}\cdot \text{Loss}_{\bm{m}}(h;\bm{\lambda}, \bm{\hat{\lambda}}_{\bm{m}}^R)
        +C'm_{b}^{-\frac{\alpha}{d-1}},
	\end{equation*}
	where $\kappa_r = \frac{C_r}{c_r}$, $\kappa_{b} = \frac{C_{b}}{c_{b}}$, 
	$C_{\bm{m}} = 3\max\{\kappa_r \sqrt{d}^{d} m_r^{\frac{1}{2}}, \kappa_{b} \sqrt{d}^{d-1} m_{b}^{\frac{1}{2}}\}$, 
	$C'$ is a universal constant that depends only on $\bm{\lambda}$, $d$, $c_{b}$, $\alpha$, $g$,
	and
	\begin{equation} \label{def:lambdas}
	\begin{split}
	    \hat{\lambda}_{r,\bm{m}}^R = \frac{3\lambda_r \sqrt{d}^{2\alpha}c_r^{-\frac{2\alpha}{d}}}{C_{\bm{m}}}\cdot m_r^{-\frac{\alpha}{d}},
        \quad
        \hat{\lambda}_{b,\bm{m}}^R = \frac{3\lambda_{b} \sqrt{d}^{2\alpha} c_{b}^{-\frac{2\alpha}{d-1}}}{C_{\bm{m}}}\cdot m_{b}^{-\frac{\alpha}{d-1}}.
	\end{split}
	\end{equation}

	For $d = 1$,
	with probability at least, $1 - \sqrt{m_r}(1-1/\sqrt{m_r})^{m_r}$, 
	we have
	\begin{equation*}
	\text{Loss}^\text{PINN}(h;\bm{\lambda})
	\le C_{\bm{m}}\cdot \text{Loss}_{\bm{m}}(h;\bm{\lambda},\bm{\hat{\lambda}}^R_{\bm{m}}),
	\end{equation*}
	where 
	$C_{\bm{m}} = 3\kappa_r m_r^{\frac{1}{2}}$, 
	$\hat{\lambda}^R_{r,\bm{m}} = \frac{\lambda_r c_r^{-2\alpha}}{\kappa_r}\cdot m_r^{-\alpha-\frac{1}{2}}$,
	$\hat{\lambda}_{b,\bm{m}}^R = 0.$


\end{theorem}
\begin{proof}
	The proof  largely follows the same lines as in \cite[Theorem $3.1$]{shin2020convergence}.
	Let $\mathcal{T}_r^{m_r}=\{\x_r^i\}_{i=1}^{m_r}$ be iid samples from $\mu_r$ on $U$
	and $\mathcal{T}_b^{m_b} = \{\x^i_b\}_{i=1}^{m_b}$ be 
	iid samples from $\mu_b$ on $\Gamma=\partial U$.
	
	By Lemma~\ref{cor-sampling},
	with probability at least   
	\begin{equation} \label{app:thm-prob}
	(1 - \sqrt{m_r}(1-1/\sqrt{m_r})^{m_r})
	(1 - \sqrt{m_{b}}(1-1/\sqrt{m_{b}})^{m_{b}}),
	\end{equation}
	$\forall \x_r \in U$ and $\forall \x_b \in \Gamma$,
	there exists $\x_r' \in \mathcal{T}_r^{m_r}$
	and $\x_b' \in \mathcal{T}_b^{m_b}$ such that 
	$\|\x_r - \x_r'\| \le \sqrt{d}c_r^{-\frac{1}{d}}m_r^{-\frac{1}{2d}}$ and $\|\x_{b} - \x_{b}'\| \le \sqrt{d}c^{-\frac{1}{d-1}}_{b}m_{b}^{-\frac{1}{2(d-1)}}$.
	Let  $\epsilon_r = \sqrt{d}c_r^{-\frac{1}{d}}m_r^{-\frac{1}{2d}}$
	and $\epsilon_{b} = \sqrt{d}c^{-\frac{1}{d-1}}_{b}m_{b}^{-\frac{1}{2(d-1)}}$.
	If we apply  Lemma B.1 in \cite{shin2020convergence} to our fully nonlinear case, then  with probability at least \eqref{app:thm-prob}, one gets
 
    \begin{align*}
    \text{Loss}^{\text{PINN}}(h;\bm{\lambda})&\le C_{\bm{m}} \cdot 
    \left[
    \text{Loss}_{\bm{m}}^{\text{PINN}}(h;\bm{\lambda})+\lambda_{r,\bm{m}}^R\cdot \big[\mathcal{F}[h]\big]_{\alpha;U}^2 + \lambda_{b,\bm{m}}^R \cdot \big[h\big]_{\alpha;\Gamma}^2
    \right] + \\&+
    3\lambda_{b}\sqrt{d}^{2\alpha} c_{b}^{-\frac{2\alpha}{d-1}}m_b^{-\frac{\alpha}{d-1}} \cdot\big[g\big]_{\alpha;\Gamma}^2,
	\end{align*}
    where
    $$
       C_{\bm{m}} = 3\max\left\{\frac{C_r}{c_r}\sqrt{d}^{d} m_r^{\frac{1}{2}},\;  \frac{C_{b}}{c_{b}} \sqrt{d}^{d-1} m_{b}^{\frac{1}{2}}\right\},
    $$
     \begin{align*}
        \lambda_{r,\bm{m}}^R = \frac{3\lambda_r \sqrt{d}^{2\alpha}c_r^{-\frac{2\alpha}{d}}m_r^{-\frac{\alpha}{d}}}{C_{\bm{m}}},
        \quad
        \lambda_{b,\bm{m}}^R = \frac{3\lambda_{b} \sqrt{d}^{2\alpha} c_{b}^{-\frac{2\alpha}{d-1}} m_{b}^{-\frac{\alpha}{d-1}}}{C_{\bm{m}}}.
    \end{align*}
    
    By taking $C' = 3\lambda_{b}\sqrt{d}^{2\alpha} c_{b}^{-\frac{2\alpha}{d-1}}\big[g\big]_{\alpha;\Gamma}^2,$ we have
    
    \begin{equation*}
        \text{Loss}^{\text{PINN}}(h;\bm{\lambda})
        \le  C_{m}\cdot \text{Loss}_{\bm{m}}(h;\bm{\lambda},\bm{\hat{\lambda}}^R_{\bm{m}})
        + C'\cdot m_{b}^{-\frac{\alpha}{d-1}},
    \end{equation*}
    This completes the proof.
\end{proof}
Let $\bm{\lambda}$ be a vector independent of $\bm{m}$ 
and 
$\bm{\lambda}_{\bm{m}}^R = (\lambda_{r,\bm{m}}^R, \lambda_{b,\bm{m}}^R)$
be a vector satisfying
\begin{equation} \label{def:lambda-condition}
    \bm{\lambda}_{\bm{m}}^R \ge \bm{\hat{\lambda}}_{\bm{m}}^R, \qquad
    \|\bm{\lambda}_{\bm{m}}^R\|_\infty = \mathcal{O}(\|\bm{\hat{\lambda}}_{\bm{m}}^R\|_\infty),
\end{equation}
where $\bm{\hat{\lambda}}_{\bm{m}}^R$ is defined in \eqref{def:lambdas}.
For a vector $v$, $\|v\|_{\infty}$ is the maximum norm, i.e., $\|v\|_{\infty} = \max_i |v_i|$.
We note that 
since $\bm{\hat{\lambda}}_{\bm{m}}^R \to 0$
as $\bm{m} \to \infty$,
the above condition implies 
$\bm{\lambda}_{\bm{m}}^R \to 0$
as $\bm{m} \to \infty$.	

Inspired by \cite{shin2020convergence}   letting $R_r(h) = \big[\mathcal{L}[h]\big]_{\alpha;U}^2$
and $R_{b}(h) = \big[\mathcal{B}[h]\big]_{\alpha;\Gamma}^2$,
we define the H\"{o}lder regularized empirical loss:
\begin{equation} \label{def:Holder-Reg-Loss}
\begin{split}
&\text{Loss}_{\bm{m}}(h;\bm{{\lambda}},\bm{{\lambda}}^R_{\bm{m}}) 
\\
&= \begin{cases}
\text{Loss}_{\bm{m}}^{\text{PINN}}(h;\bm{{\lambda}})
+\lambda_{r,\bm{m}}^R
\big[\mathcal{F}[h]\big]_{\alpha;U}^2
+\lambda_{b,\bm{m}}^R
\big[h\big]_{\alpha;\Gamma}^2, & \text{if } d \ge 2 \\[12pt]
\text{Loss}_{\bm{m}}^{\text{PINN}}(h;\bm{{\lambda}})
+ \lambda_{r,\bm{m}}^R
\big[\mathcal{F}[h]\big]_{\alpha;U}^2,  &\text{if } d = 1
\end{cases}
\end{split}
\end{equation}
where $\bm{\lambda}_{\bm{m}}^R$ are vectors satisfying \eqref{def:lambda-condition}.
We note that the H\"{o}lder regularized loss \eqref{def:Holder-Reg-Loss} is greater than or equal to the empirical loss shown in Theorem~\ref{thm:gen} (assuming $R_r(h) = \big[\mathcal{L}[h]\big]_{\alpha;U}^2$
and $R_{b}(h) = \big[h\big]_{\alpha;\Gamma}^2$).
Since $\big[\mathcal{F}[h]\big]_{\alpha;U}^2$
and $\big[h\big]_{\alpha;\Gamma}^2$ do not depend on the training data
and $\hat{\lambda}_{r,\bm{m}}^R, \hat{\lambda}_{b,\bm{m}}^R \to 0$ as $m_r, m_{b} \to \infty$,
this suggests that the more data we have, the less regularization is needed.

According to \cite{shin2020convergence} we make the following assumptions  on the classes of neural networks
for the minimization problems \eqref{def:problem}.
\begin{assumption} \label{assumption:convergence}
	Let $k$ be the highest order of the derivative shown in the PDE \eqref{PDE}.
	For some $0 < \alpha \le 1$, 
	let  $g \in C^{0,\alpha}(\Gamma)$.
	\begin{itemize}

	    \item[$\bullet$] 
        For each $\bm{m}$,
	    let $\mathcal{H}_{\bm{m}}$ be a class of neural networks
	    in $C^{k,\alpha}(U)\cap C^{0,\alpha}(\overline{U})$
	    such that
    	for any $h \in \mathcal{H}_{\bm{m}}$, $\mathcal{F}[h] \in C^{0,\alpha}(U)$ 
	    and $h \in C^{0,\alpha}(\Gamma)$.
		\item[$\bullet$]  For each $\bm{m}$, $\mathcal{H}_{\bm{m}}$ contains a network $u_{\bm{m}}^*$ 
		satisfying 
		$\text{Loss}_{\bm{m}}^{\text{PINN}}(u_{\bm{m}}^*;\bm{\lambda}) = 0$.
		 \item[$\bullet$] 
		Moreover we assume, 
		$$
		\sup_{\bm{m}} \big[\mathcal{F}[u_{\bm{m}}^*]\big]_{\alpha;U}<+\infty, \quad \sup_{\bm{m}} [u_{\bm{m}}^*]_{\alpha;\Gamma} < +\infty.
		$$
	\end{itemize}
\end{assumption}

All the assumptions are essential.  the Assumption~\ref{assumption:convergence} needed to
guarantee the uniform equicontinuity of subsequences.
All assumptions hold automatically if $\mathcal{H}_{\bm{m}}$ contains the solution to the PDE for all $\bm{m}$.
For example, \cite{Darbon_19_NNsHJ,Darbon_20_NNsHJ} showed that the solution 
to some Hamilton-Jacobi PDEs can be exactly represented by neural networks.
{
\begin{remark}
    The choice of classes of neural networks $\mathcal{H}_{\bm{m}}$
    may depend on many factors
    including 
    the underlying PDEs, 
    the training data,
    and the network architecture.
    Assumption~\ref{assumption:convergence}
    provides a set of conditions 
    for neural networks 
    for the purpose of our analysis.
\end{remark}}

For the rest of this paper, we use the following notation.
When the number of the initial/boundary training data points $m_{b}$ is completely determined by the number of residual points $m_r$ (e.g. $m_r^{d-1} = \mathcal{O}(m_{b}^{d})$),
the vector of the number of training data $\bm{m}$ depends only on $m_r$. 
In this case, we simply write $\mathcal{H}_{\bm{m}}$, $\text{Loss}_{\bm{m}}$ 
as $\mathcal{H}_{m_r}$, $\text{Loss}_{m_r}$, respectively.

Again following \cite{shin2020convergence}, we can further extend the result stating that minimizers of the loss \eqref{def:PINN-loss}  yield a small expected PINN loss in the nonlinear case. Therefore, in our case, the theorem can be formulated as follows:

\begin{theorem} \label{thm:conv-loss}
	Suppose Assumptions~\ref{assumption:data-dist} and ~\ref{assumption:convergence} hold.
	Let $m_r$ and $m_{b}$ be the number of iid samples from 
	$\mu_r$ and $\mu_{b}$, respectively,
	and $m_r = \mathcal{O}(m_{b}^{\frac{d}{d-1}})$.
	Let $\bm{\lambda}_{m_r}^R$ be a vector satisfying 
	\eqref{def:lambda-condition}.
	Let $h_{m_r} \in \mathcal{H}_{m_r}$ be a minimizer of the H\"{o}lder regularized loss $\text{Loss}_{m_r}(\cdot;\bm{\lambda}, \bm{\lambda}_{m_r}^R)$ \eqref{def:Holder-Reg-Loss}.
	Then the following holds.
	\begin{itemize}
	    \item[$\bullet$] With probability at least $(1 - \sqrt{m_r}(1-c_r/\sqrt{m_r})^{m_r})(1 - \sqrt{m_{b}}(1-c_{b}/\sqrt{m_{b}})^{m_{b}})$
	    over iid samples,
	\begin{equation*}
	\text{Loss}^{\text{PINN}}(h_{m_r};\bm{\lambda}) = \mathcal{O}(m_r^{-\frac{\alpha}{d}}).
	\end{equation*} 
	    \item[$\bullet$] With probability 1 over iid samples, 
	\begin{equation} \label{thm:l2-convergence}
	\lim_{m_r \to \infty} \mathcal{F}[h_{m_r}] = 0 \text{ in } C^0(U),
	\quad
	\lim_{m_r \to \infty} h_{m_r} = g \text{ in } C^0(\Gamma).
	\end{equation}
	\end{itemize}
\end{theorem}
\begin{proof}
		Here for the sake of completeness  we give a proof by following \cite[Theorem $3.2$]{shin2020convergence}. 
 Suppose $m_r = \mathcal{O}(m_{b}^{\frac{d}{d-1}})$. 
	It then can be checked that 
	$\hat{\lambda}_{r,\bm{m}}^R = \hat{\lambda}_{b,\bm{m}}^R =\mathcal{O}(m_r^{-\frac{1}{2}-\frac{\alpha}{d}})$,
	where $\hat{\lambda}_{r,\bm{m}}$ and $\hat{\lambda}_{b,\bm{m}}$ are defined in \eqref{def:lambdas}.
	Let $\bm{\lambda}$ be a vector independent of $\bm{m}$
	and $\bm{\lambda}_{\bm{m}}^R$ be a vector satisfying \eqref{def:lambda-condition}.
	
	Let $h_{m} \in \mathcal{H}_{m}$ be a function that minimizes $\text{Loss}_{\bm{m}}(\cdot;\bm{\lambda}, \bm{\lambda}_{\bm{m}}^R)$.
	Recalling that $u_{\bm{m}}^* \in \mathcal{H}_{\bm{m}}$,  we write
	\begin{equation} \label{app:eqn-loss-convg}
	\begin{split}
	&
	\min\{\lambda_{r,\bm{m}}^R,\lambda_{b,\bm{m}}^R\} \left(R_r(h_{m}) + R_{b}(h_{m})\right) 
	\le \text{Loss}_{m}(h_{m};\bm{\lambda}, \bm{\lambda}_{\bm{m}}^R)
	\\
	&\le 
	\text{Loss}_{m}(u_{\bm{m}}^*;\bm{\lambda}, \bm{\lambda}_{\bm{m}}^R)
	\le
	\text{Loss}_{m}(u_{\bm{m}}^*;\bm{\lambda}, \bm{0}) + 
	\|\bm{\lambda}_{\bm{m}}^R\|_{\infty}\left(R_r(u_{\bm{m}}^*) + R_{b}(u_{\bm{m}}^*)\right). 
	\end{split}
	\end{equation}
	Since $\bm{\lambda}_{\bm{m}}^R \ge \hat{\bm{\lambda}}_{\bm{m}}^R$
	and $\|\bm{\lambda}_{\bm{m}}^R\|_\infty = \mathcal{O}(\|\hat{\bm{\lambda}}_{\bm{m}}^R\|_\infty)$, we have 
	$\frac{\max_j (\bm{\lambda}_{\bm{m}}^R)_j}{\min_j (\bm{\lambda}_{\bm{m}}^R)_j} =\mathcal{O}(1)$.
	Let $R^* = \sup_{\bm{m}} (R_r(u_{\bm{m}}^*) + R_b(u_{\bm{m}}^*))$.
	By the third assumption in 
 Assumption \ref{assumption:convergence}, we have $R^* < \infty$.

 Let us write $m_r$ as $m$ for the sake of simplicity.
	We then have $R_r(h_{m}), R_{b}(h_{m})\le \mathcal{O}(R^*)$ for all $m$.
	Since $R_r(h_{m})=
	\big[\mathcal{F}[h_{m}]\big]_{\alpha;U}^2$
	and $R_{b}(h_{m})=\big[h_{m}\big]_{\alpha;\Gamma}^2$,
	the H\"{o}lder coefficients of $\mathcal{F}[h_{m}]$ and $h_{m}$ are uniformly bounded above.
	With the first assumption in Assumption \ref{assumption:convergence},
    $\{\mathcal{F}[h_{m}]\}$ and $\{h_{m}\}$
    are uniformly bounded and uniformly equicontinuous sequences of functions in
    $C^{0,\alpha}(U)$ and $C^{0,\alpha}(\Gamma)$, respectively.
	By invoking the Arzela-Ascoli Theorem, 
	there exists a subsequence $h_{m_j}$
	and functions $G \in C^{0,\alpha}(U)$ and $B \in C^{0,\alpha}(\Gamma)$
	such that $\mathcal{F}[h_{m_j}] \to G$ 
	and $h_{m_j}\to B$ in $C^{0}(U)$ and $C^{0}(\Gamma)$, respectively, as $j \to \infty$.

    Since $\text{Loss}_m(h_m;\bm{\lambda},\bm{\lambda}_{\bm{m}}^R) = \mathcal{O}(m_r^{-\frac{1}{2}-\frac{\alpha}{d}})$,
	by combining it with Theorem~\ref{thm:gen}, we have that 
	with probability at least $(1 - \sqrt{m_r}(1-c_r/\sqrt{m_r})^{m_r})(1 - \sqrt{m_{b}}(1-c_{b}/\sqrt{m_{b}})^{m_{b}})$,
	\begin{equation*}
	\text{Loss}(h_{m};\bm{\lambda},\bm{0}) = \mathcal{O}(m^{-\frac{\alpha}{d}}). 
	\end{equation*}

	Hence, the probability of
	$\lim_{m \to \infty} 
	\text{Loss}(h_{m};\bm{\lambda}, \bm{0}) = 0$
	is one. Thus, with probability 1, 
	\begin{align*}
	0 &= \lim_{j \to \infty} \text{Loss}(h_{m_j};\bm{\lambda})
	\\
	&= \lim_{j \to \infty} \lambda_r \int_{U} \|\mathcal{F}[h_{m_j}](\x_r)\|^2d\mu_r(\x_r) +  \lambda_{b}\int_{\Gamma} \|h_{m_j}(\x_{b}) - g(\x_{b}) \|^2 d\mu_{b}(\x_{b})
	\\
	&=\lambda_{f}\int_{U} \|G(\x_r)\|^2d\mu_r(\x_r) + 
	\lambda_{b} \int_{\Gamma} \|B(\x_{b}) - g(\x_{b}) \|^2 d\mu_{b}(\x_{b}),
	\end{align*}
	which shows that 
	$G = f$ in $L^2(U;\mu_r)$ and $B = g$ in $L^2(\Gamma;\mu_{b})$.
	Note that since $\mathcal{F}[h_{m_j}]$ and $h_{m_j}$ are uniformly bounded above and uniformly converge to $G$ and $B$, respectively,
	the third equality holds by Lebesgue's Dominated Convergence Theorem.
	Since the subsequence was arbitrary,
	we conclude that $\mathcal{F}[h_{m}] \to 0$ in $L^2(U;\mu_r)$ and $h_m \to g$
	in $L^2(\Gamma;\mu_{b})$ as $m \to \infty$.
	Furthermore, since $\{\mathcal{F}[h_{m}]\}$ and $\{h_m\}$ are equicontinuous, then
	we have  uniform convergence, i.e. convergence in  $C^{0}$. The proof of the theorem is completed.  
\end{proof}

Theorem~\ref{thm:conv-loss} shows that the expected PINN loss \eqref{def:PINN-loss-2} at minimizers of the  loss \eqref{def:PINN-loss} converges to zero.
And more importantly, it shows the uniform convergences of $\mathcal{F}[h_{m_r}] \to 0$ and $h_{m_r}\to g$ as ${m_r} \to \infty$.
These results are crucial, however,
it is insufficient to claim the convergence of $h_{m_r}$ to the solution to the PDE \eqref{PDE}.
To this aim, the authors in \cite{shin2020convergence} utilized the well-known Schauder estimates to establish the convergence of $h_{m_r}$ for linear PDEs. In our context, however, we are dealing with fully nonlinear PDE's and therefore such approach cannot be applied.

The main result of the paper reads as follows.
\begin{theorem}
Suppose Assumptions~\ref{assumption:data-dist} and ~\ref{assumption:convergence} hold. Let $m_r$ and $m_b$ be the number of iid samples from $\mu_r$ and $\mu_b,$ respectively, and $m_r =\mathcal{O}(m_b^{\frac{d}{d-1}}).$ We also set $U_{m_r}=\mathcal{T}^{m_r}_r\cup\mathcal{T}^{m_b}_b,$ where $\mathcal{T}_r^{m_r}=\{\x_r^i\}_{i=1}^{m_r}$ be iid samples from $\mu_r$ on $U$ and $\mathcal{T}_b^{m_b} = \{\x^i_b\}_{i=1}^{m_b}$ be iid samples from $\mu_b$ on $\Gamma$.
Let $\bm{\lambda}_{m_r}^R$ be a vector satisfying 
	\eqref{def:lambda-condition} and
 $h_{m_r} \in \mathcal{H}_{m_r}$ be a minimizer of the H\"{o}lder regularized loss $\text{Loss}_{m_r}(\cdot;\bm{\lambda}, \bm{\lambda}_{m_r}^R)$ \eqref{def:Holder-Reg-Loss}.
If  for $0<\alpha\leq 1$ given in Assumption \ref{assumption:convergence}, we have
\begin{equation}\label{uniform_bdd_minimizers}
    \sup_{{m_r}} \big[h_{m_r}\big]_{\alpha;U_{m_r}}<+\infty,
\end{equation}
then with probability 1 over iid samples
 $$
 \underset{m_r\to\infty}{lim}h_{m_r}=u^*\;\text{ in }\; C^0(\overline{U}),
 $$
 where $u^*$ is a unique viscosity solution to the system \eqref{PDE}.
\end{theorem}

\begin{proof}
Observe that  the condition \eqref{uniform_bdd_minimizers} implies  the existence of a universal constant  $C>0$ independent of $m_r$  such that for every $\x,\y\in {U}_{m_r}$ 
$$
|h_{m_r}(\x)- h_{m_r}(\y)|\leq C||\x-\y||^\alpha.
$$
Let $p_{m_r}: \overline{U}\rightarrow\mathcal{T}^{m_r}_r\cup\mathcal{T}^{m_b}_b$ be the closest point projection such that

\begin{equation}
||\x-p_{m_r}(\x)||= 
\begin{cases}
   \underset{\y\in \mathcal{T}^{m_r}_r}{\mbox{min}}||\x-\y||\;\; \mbox{if}\;\; \x\in U\\
   \underset{\y\in \mathcal{T}^{m_b}_b}{\mbox{min}}||\x-\y||\;\; \mbox{if} \;\; \x\in \Gamma.
\end{cases}
\end{equation}

Define by $v_{m_r}(\x)=h_{m_r}(p_{m_r}(\x)).$ It is apparent that $v_{m_r}(\x)=h_{m_r}(\x)$ for every $\x\in \mathcal{T}^{m_r}_r\cup\mathcal{T}^{m_b}_b.$ 
Again recalling Lemma~\ref{cor-sampling}, we get with probability at least   
	\begin{equation} 
	(1 - \sqrt{m_r}(1-1/\sqrt{m_r})^{m_r})
	(1 - \sqrt{m_{b}}(1-1/\sqrt{m_{b}})^{m_{b}}),
	\end{equation}
	$\forall \x_r \in U$ and $\forall \x_b \in \Gamma$,
	there exists $\x_r' \in \mathcal{T}_r^{m_r}$
	and $\x_b' \in \mathcal{T}_b^{m_b}$ such that 
	$\|\x_r - \x_r'\| \le \sqrt{d}c_r^{-\frac{1}{d}}m_r^{-\frac{1}{2d}}$ and $\|\x_{b} - \x_{b}'\| \le \sqrt{d}c^{-\frac{1}{d-1}}_{b}m_{b}^{-\frac{1}{2(d-1)}}$.
	By letting $\bm{\varepsilon}_r = \sqrt{d}c_r^{-\frac{1}{d}}m_r^{-\frac{1}{2d}}$
	and $\bm{\varepsilon}_{b} = \sqrt{d}c^{-\frac{1}{d-1}}_{b}m_{b}^{-\frac{1}{2(d-1)}}$, 
	it follows $||\x_r-p_{m_r}(\x_r)||\leq \bm{\varepsilon}_r$ whenever $\x_r\in U,$ and $||\x_b-p_{m_r}(\x_b)||\leq \bm{\varepsilon}_b$ whenever $\x_b\in \Gamma.$

Thus, $\forall \x,\y \in \overline{U}$
\begin{align*}
|v_{m_r}(\x)- v_{m_r}(\y)|&\leq C||p_{m_r}(\x)- p_{m_r}(\y)||^\alpha\\
&=C||p_{m_r}(\x)-\x +\y-p_{m_r}(\y) + (\x-\y)||^\alpha\\
&\leq C\left(||\x-\y||^\alpha + 2\max(\bm{\varepsilon}^\alpha_r,\bm{\varepsilon}^\alpha_b)\right).
\end{align*}
We  use a variant of the Arzelà-Ascoli Theorem (see the appendix in \cite{calder2015pde}) to show that there exists a subsequence, which we again denote by $v_{m_r},$ and a 
H\"{o}lder  continuous function $u^* \in C^{0,\alpha}(U)$ such that $v_{m_r}\to u^*$ uniformly on $\overline{U}$ as $m_r\to\infty.$
Since  $v_{m_r}(\x)=h_{m_r}(\x)$ for every $\x\in \mathcal{T}^{m_r}_r\cup\mathcal{T}^{m_b}_b,$ then we have 
\begin{equation}\label{uniform_conv}
\lim_{m_r\to\infty}\max_{\mathcal{T}^{m_r}_r\cup\mathcal{T}^{m_b}_b}|h_{m_r}(\x)-u^*(\x)|=0.
\end{equation}
We claim that $u^*$ is a viscosity solution to \eqref{PDE}.
Once this is proved, we can apply the same argument to any subsequence of $h_{m_r}$ to show that the entire sequence converges uniformly to $u^*.$
Let $\x_0\in U$ and $\varphi\in C^\infty(\mathbb{R}^d)$ such that $u^*-\varphi$ has a strict local maximum at the point $\x_0$ and $\nabla \varphi(\x_0)\neq 0.$

Following \eqref{uniform_conv} there exists a sequence of points $\x_{m_r}\in \mathcal{T}^{m_r}_r\cup\mathcal{T}^{m_b}_b,$ such that $h_{m_r}-\varphi$ attains its local maximum at $\x_{m_r}$ and $\x_{m_r}\to \x_0$ as $r\to \infty.$ Moreover, $h_{m_r}(\x_{m_r})\to u^*(\x_0).$  Since $h_{m_r}$ and $\varphi$ are smooth functions (due to sufficiently smooth activation function in PINN), then  at the point $\x_{m_r}$ holds:
$$
D(h_{m_r}-\varphi)(\x_{m_r})=0 \;\;\mbox{and} \;\; D^2 (h_{m_r}-\varphi)(\x_{m_r})\leq 0.
$$
Thus applying the degenerate ellipticity property we obtain
$$
 F\left(\x_{m_r}, h_{m_r}(\x_{m_r}), D\varphi(\x_{m_r}), D^2\varphi(\x_{m_r}) \right)\leq F\left(\x_{m_r}, h_{m_r}(\x_{m_r}), D h_{m_r}(\x_{m_r}), D^2 h_{m_r}(\x_{m_r}) \right),
$$
or in short 
$$
 F\left(\x_{m_r}, h_{m_r}(\x_{m_r}), D\varphi(\x_{m_r}), D^2\varphi(\x_{m_r}) \right)\leq \mathcal{F}[h_{m_r}](\x_{m_r}).
$$
Due to Theorem \ref{thm:conv-loss} we know that  $\mathcal{F}[h_{m_r}]$ converges to $0$ uniformly as $m_r\to\infty.$ Hence 
$$
F\left(\x_{m_r}, h_{m_r}(\x_{m_r}), D\varphi(\x_{m_r}), D^2\varphi(\x_{m_r}) \right)\leq \mathcal{F}[h_{m_r}](\x_{m_r})\to 0\;\;\mbox{as}\;\; m_r\to\infty,
$$
which leads to
$$
F\left(\x_0, u^*(\x_0), D\varphi(\x_0), D^2\varphi(\x_0) \right)\leq 0.
$$
Therefore $u^*$ is a viscosity subsolution of \eqref{PDE}. Similarly, we can check  the supersolution property.
This completes the proof.
\end{proof}

\section*{Acknowledgment}
A. Arakelyan was supported by the Science Committee of RA (Research project № 24IRF-1A001).

%
%
%
\bibliographystyle{splncs04}
\bibliography{biblio}

\begin{thebibliography}{10}
\providecommand{\url}[1]{\texttt{#1}}
\providecommand{\urlprefix}{URL }
\providecommand{\doi}[1]{https://doi.org/#1}

\bibitem{Baker_19_DCworkshop}
Baker, N., Alexander, F., Bremer, T., Hagberg, A., Kevrekidis, Y., Najm, H.,
  Parashar, M., Patra, A., Sethian, J., Wild, S., AL., E.: Workshop report on
  basic research needs for scientific machine learning: Core technologies for
  artificial intelligence. Tech. rep., USDOE Office of Science (SC),
  Washington, DC (United States) (2019)

\bibitem{Berg_18_Unified}
Berg, J., Nystr{\"o}m, K.: A unified deep artificial neural network approach to
  partial differential equations in complex geometries. Neurocomputing
  \textbf{317},  28--41 (2018)

\bibitem{MR1351007}
Caffarelli, L.A., Cabr{\'e}, X.: Fully nonlinear elliptic equations, American
  Mathematical Society Colloquium Publications, vol.~43. American Mathematical
  Society, Providence, RI (1995)

\bibitem{cai2021physics}
Cai, S., Mao, Z., Wang, Z., Yin, M., Karniadakis, G.E.: Physics-informed neural
  networks (pinns) for fluid mechanics: A review. Acta Mechanica Sinica
  \textbf{37}(12),  1727--1738 (2021)

\bibitem{calder2019consistency}
Calder, J.: Consistency of lipschitz learning with infinite unlabeled data and
  finite labeled data. SIAM Journal on Mathematics of Data Science
  \textbf{1}(4),  780--812 (2019)

\bibitem{calder2015pde}
Calder, J., Esedoglu, S., Hero, A.O.: A pde-based approach to nondominated
  sorting. SIAM Journal on Numerical Analysis  \textbf{53}(1),  82--104 (2015)

\bibitem{MR732102}
Crandall, M.G., Evans, L.C., Lions, P.L.: Some properties of viscosity
  solutions of {H}amilton-{J}acobi equations. Trans. Amer. Math. Soc.
  \textbf{282}(2),  487--502 (1984). \doi{10.2307/1999247},
  \url{http://dx.doi.org/10.2307/1999247}

\bibitem{MR1118699}
Crandall, M.G., Ishii, H., Lions, P.L.: User's guide to viscosity solutions of
  second order partial differential equations. Bull. Amer. Math. Soc. (N.S.)
  \textbf{27}(1),  1--67 (1992)

\bibitem{MR690039}
Crandall, M.G., Lions, P.L.: Viscosity solutions of {H}amilton-{J}acobi
  equations. Trans. Amer. Math. Soc.  \textbf{277}(1),  1--42 (1983).
  \doi{10.2307/1999343}, \url{http://dx.doi.org/10.2307/1999343}

\bibitem{Darbon_19_NNsHJ}
Darbon, J., Langlois, G.P., Meng, T.: Overcoming the curse of dimensionality
  for some hamilton--jacobi partial differential equations via neural network
  architectures. Research in the Mathematical Sciences  \textbf{7}(3), ~20
  (2020)

\bibitem{Darbon_20_NNsHJ}
Darbon, J., Meng, T.: On some neural network architectures that can represent
  viscosity solutions of certain high dimensional hamilton--jacobi partial
  differential equations. Journal of Computational Physics  \textbf{425},
  109907 (2021)

\bibitem{Dissanayake_94_ANN-PDE}
Dissanayake, M., Phan-Thien, N.: Neural-network-based approximations for
  solving partial differential equations. Communications in Numerical Methods
  in Engineering  \textbf{10}(3),  195--201 (1994)

\bibitem{eivazi2024physics}
Eivazi, H., Wang, Y., Vinuesa, R.: Physics-informed deep-learning applications
  to experimental fluid mechanics. Measurement science and technology
  \textbf{35}(7),  075303 (2024)

\bibitem{Finlay_18_Lipschitz}
Finlay, C., Calder, J., Abbasi, B., Oberman, A.: Lipschitz regularized deep
  neural networks generalize and are adversarially robust. arXiv preprint
  arXiv:1808.09540  (2018)

\bibitem{Friedman_08_ParabolicPDEs}
Friedman, A.: Partial differential equations of parabolic type. Courier Dover
  Publications (2008)

\bibitem{Gilbarg_15_EllipticPDEs}
Gilbarg, D., Trudinger, N.S.: Elliptic partial differential equations of second
  order. Springer (2015)

\bibitem{griewank2008evaluating}
Griewank, A., Walther, A.: Evaluating derivatives: principles and techniques of
  algorithmic differentiation. SIAM (2008)

\bibitem{grossmann2024can}
Grossmann, T.G., Komorowska, U.J., Latz, J., Sch{\"o}nlieb, C.B.: Can
  physics-informed neural networks beat the finite element method? IMA Journal
  of Applied Mathematics  \textbf{89}(1),  143--174 (2024)

\bibitem{HanE_18_DLSPDEs}
Han, J., Jentzen, A., E, W.: Solving high-dimensional partial differential
  equations using deep learning. Proceedings of the National Academy of
  Sciences  \textbf{115}(34),  8505--8510 (2018)

\bibitem{Houska_19_Global}
Houska, B., Chachuat, B.: Global optimization in hilbert space. Mathematical
  programming  \textbf{173}(1-2),  221--249 (2019)

\bibitem{Lagaris_98_ANN-ODE-PDE}
Lagaris, I.E., Likas, A., Fotiadis, D.I.: Artificial neural networks for
  solving ordinary and partial differential equations. IEEE transactions on
  Neural Networks  \textbf{9}(5),  987--1000 (1998)

\bibitem{Lagaris_00_ANN-Irregular}
Lagaris, I.E., Likas, A.C., Papageorgiou, G.D.: Neural-network methods for
  boundary value problems with irregular boundaries. IEEE Transactions on
  Neural Networks  \textbf{11}(5),  1041--1049 (2000)

\bibitem{Lecun_Nature15_DeepLearning}
LeCun, Y., Bengio, Y., Hinton, G.: Deep learning. Nature  \textbf{521}(7553),
  436--444 (2015)

\bibitem{Lu_19_Deepxde}
Lu, L., Meng, X., Mao, Z., Karniadakis, G.E.: Deepxde: A deep learning library
  for solving differential equations. arXiv preprint arXiv:1907.04502  (2019)

\bibitem{Mao_20_HighSpeedFlows}
Mao, Z., Jagtap, A.D., Karniadakis, G.E.: Physics-informed neural networks for
  high-speed flows. Computer Methods in Applied Mechanics and Engineering
  \textbf{360},  112789 (2020)

\bibitem{Pang_SISC19_fPINNs}
Pang, G., Lu, L., Karniadakis, G.E.: f{PINN}s: Fractional physics-informed
  neural networks. SIAM Journal on Scientific Computing  \textbf{41}(4),
  A2603--A2626 (2019)

\bibitem{Raissi_19_PINNs}
Raissi, M., Perdikaris, P., Karniadakis, G.E.: Physics-informed neural
  networks: A deep learning framework for solving forward and inverse problems
  involving nonlinear partial differential equations. Journal of Computational
  Physics  \textbf{378},  686--707 (2019)

\bibitem{Raissi_Nature20_HFM}
Raissi, M., Yazdani, A., Karniadakis, G.E.: Hidden fluid mechanics: Learning
  velocity and pressure fields from flow visualizations. Science
  \textbf{367}(6481),  1026--1030 (2020)

\bibitem{reddy1993introduction}
Reddy, J.N.: An introduction to the finite element method. New York
  \textbf{27}, ~14 (1993)

\bibitem{ren2023class}
Ren, H., Meng, X., Liu, R., Hou, J., Yu, Y.: A class of improved fractional
  physics informed neural networks. Neurocomputing  \textbf{562},  126890
  (2023)

\bibitem{shin2020convergence}
Shin, Y., Darbon, J., Karniadakis, G.E.: On the convergence of physics informed
  neural networks for linear second-order elliptic and parabolic type pdes.
  Communications in Computational Physics  \textbf{28}(5),  2042--2074 (2020)

\bibitem{Sirignano_JCP18_DGM}
Sirignano, J., Spiliopoulos, K.: D{GM}: A deep learning algorithm for solving
  partial differential equations. Journal of Computational Physics
  \textbf{375},  1339--1364 (2018)

\bibitem{Song_19_fPINNs}
Song, F., Pange, G., Meneveau, C., Karniadakis, G.E.: Fractional
  physical-inform neural networks (f{PINN}s) for turbulent flows. Bulletin of
  the American Physical Society  (2019)

\bibitem{strang2008analysis}
Strang, G., Fix, G.: An Analysis of the Finite Element Methods, and
  Engineering. SIAM (2008)

\bibitem{Zhang_19_SPDE_PINNs}
Zhang, D., Guo, L., Karniadakis, G.E.: Learning in modal space: Solving
  time-dependent stochastic pdes using physics-informed neural networks. SIAM
  Journal on Scientific Computing  \textbf{42}(2),  A639--A665 (2020)

\end{thebibliography}
\end{document}